\documentclass[11pt]{amsart}
\usepackage{pdfsync}
\usepackage{amsmath,amssymb}

\renewcommand{\epsilon}{\varepsilon}
\renewcommand{\phi}{\varphi}

\newcommand{\su}{\subseteq}
\newcommand{\rest}{\restriction}
\renewcommand{\a}{\alpha}
\renewcommand{\b}{\beta}
\newcommand{\g}{\gamma}

\newcommand{\n}{\nu}
\renewcommand{\d}{\delta}
\renewcommand{\l}{\lambda}
\renewcommand{\k}{\kappa}

\newcommand{\om}{\omega}

\newcommand{\lng}{\langle}
\newcommand{\rng}{\rangle}
\newcommand{\ov}{\overline}
\newcommand{\sm}{\setminus}

\newcommand{\Licf}{{\operatorname {\chi_{\ell CF} }}}

\newcommand{\acc}{{\operatorname {acc}}}

\newcommand{\cf}{{\operatorname {cf}}}

\newcommand{\imply}{\Rightarrow}

\newtheorem{theorem}{Theorem}[section]
\newtheorem{corollary}[theorem]{Corollary}
\newtheorem{lemma}[theorem]{Lemma}

\newtheorem{definition}[theorem]{Definition}
\newtheorem{claim}[theorem]{Claim}

\author{Menachem Kojman}

\thanks{The author was supported by a fellowship from the
  Institute for Advanced Study, Princeton, NJ, while working on this
  research.}

\address{Department of Mathematics\\
Ben-Gurion University of the Negev\\
P.O.B. 653 \\62qyor
Be'er Sheva\\
84105 Israel}

\email{kojman@math.bgu.ac.il}


\keywords{Shelah's revised GCH, ZFC, family of sets, Property
  B, Miller's theorem, disjoint refinement, essentially disjoint
  family, conflict-free number, filtrations,
  Skolem-L\"owenhem theorem}

\begin{document}
\title{Splitting families of sets in ZFC}
\maketitle

\begin{abstract} 
  Miller's 1937 splitting theorem was proved for pairs of cardinals
  $(\n,\rho)$ in which $n$ is finite and $\rho$ is infinite.  An
  extension of Miller's theorem is proved here in ZFC for pairs of
  cardinals $(\nu,\rho)$ in which $\nu$ is arbitrary and $\rho\ge
  \beth_\om(\nu)$.  The proof uses a new general method that is based
  on Shelah's revises Generalized Continuum Hypothesis theorem.

  Upper bounds on conflict-free coloring numbers of families of sets
  and a general comparison theorem follow as corollaries of the main
  theorem. Other corollaries eliminate the use of additional axioms
  from splitting theorems due to Erd\H os, Hajnal, Komj\'ath, Juh\'
  asz and Shelah.
\end{abstract}

\section{Introduction}
Saharon Shelah's work on pcf theory\cite{CA} transformed cardinal
arithmetic by revealing a robust structure of small products of
regular cardinals, from which absolute upper bounds followed for
singular cardinal exponentiation, in contrast to regular cardinals
exponentiation, for which no upper bounds exist by Cohen's and
Easton's results.  Shelah described the state of
affairs in 1992 in an expository article \cite{ca-skeptics}:
\begin{quote}  ``\dots the theory of cardinal arithmetic involves
two quite different aspects, one of which is totally independent of
the usual axioms of set theory, while the other is quite amenable to
investigation on the basis of ZFC.''
\end{quote}

In 2000, Shelah made an important breakthrough in cardinal arithmetic
beyond pcf theory and beyond exponentiation \cite{sh:460} by
introducing a new arithmetic function, the \emph{revised power}
$\l^{[\k]}$ (definition \ref{rp} below), and proving the equality
$\l^{[\k]}=\l$ for all $\l\ge\mu$, regular or singular, with all but a
bounded set of regular $\k<\mu$, for an arbitrary strong limit
cardinal $\mu$. This was the first time since the equality $\l^n=\l$
was proved for infinite $\l$ and finite $n>0$ that cardinal arithmetic
equations were shown to hold absolutely in an end-segment of the
cardinals. Shelah explained in the introduction to \cite{sh:460} why
he considered this discovery to be a solution of Hilbert's first
problem, necessarily reinterpreted after Cohen's proof of the
independence of the continuum hypothesis as what are the rules of
cardinals arithmetic, rather than what is the size of the continuum.


We present here a method, based on Shelah's revised GCH theorem, for
proving combinatorial theorems for an arbitrary infinite cardinal
$\nu$ in all sufficiently large cardinalities, in a similar way to
that in which theorems for a finite $n$ are proved in all infinite
cardinals. This approach was first used in \cite{NogaShelah} to
extends to infinite graphs results about finite graph colorings. Here
a general methods is presented.  The particular applications of the
method here concern the old subject of splitting families of sets.

In Section 2 the arithmetic properties of the density function
$\mathcal D(\l,\k_1,\k_2)$ for $\k_1\le\k_2\le \l$ are
handled. $\mathcal D(\l,\k_1,\k_2)$ is the least number of
$\k_1$-subsets of $\l$ required to enter every $\k_2$-subset of $\l$, 
and is clearly bounded from above by the exponent $\l^{\k_1}$. Unlike
the exponent function $\l^{\k_1}$, the arithmetic of $\mathcal
D(\l,\nu,\k)$ is robust. It is not affected by Cohen or by Easton
forcing, the equality $\mathcal D(\l,\nu,\nu^+)=\l$ for all $\l\ge
2^\nu$ holds easily in models of set theory which obey mild
restrictions on cardinal exponentiation, and, most importantly,
$\mathcal D(\l,\nu,\k)=\l$ for all sufficiently large $\l\ge\k$ in
\emph{all} models of ZFC as a consequence of Shelah's revised GCH
theorem.

In Section 3 an asymptotic filtration theorem for order-reversing set
operations is obtained. This theorem is an absolute weak infinitary
generalization of the well-known downwards Skolem-L\"owenheim
theorem. Instead of using the equation $\l^\n=\l$ --- which fails
periodically in the class of all infinite cardinals --- the proof
relies on the equations which govern the  density function.

\medbreak Miller \cite{miller} proved in 1937 a splitting theorem for
arbitrarily large $\rho$-uniform families of sets that satisfy the
condition $C(\rho^+,n)$ when $\rho\ge \aleph_0$ and $n$ is finite. The
condition $C(\rho^+,n)$ is that the intersection of any subfamily of
size at least $\rho^+$ has cardinality smaller than $n$. Miller's
proof relied on the equation $\l^n=\l$ and essentially used the
Skolem-L\"owenheim argument (although Miller most likely did not know
about Skolem's work at the time) to obtain filtrations. Theorem
\ref{main} below strengthens Miller's theorem and extends it to any
cardinal $\nu$ in place of the finite $n$ and any $\rho\ge
\beth_\om(\nu)$ in place of the infinite $\rho$ in Miller's
$C(\rho^+,n)$ condition. For a finite $n$, the term $\beth_\om(n)$ is
indeed $\aleph_0$.  The required filtrations are gotten by using
Section 2.

Several corollaries about splitting, comparing and coloring families
of sets are gotten from Theorem \ref{main}. Three of these results
were proved earlier by assuming additional axioms. Erd\H os and Hajnal
proved property B from $C(\nu, \rho^+)$ for $\rho>\nu^+$ from the
Generalized Continuum Hypothesis \cite{eh}, Komj\'ath obtained
disjoint refinements and essential disjointness from the same axiom
\cite{komclose} and Hajnal, Juh\'asz and Shelah obtained essential
disjointness from a weak form of the Singlar Cardinals
Hypothesis \cite{hjs}.

\section{ Arithmetic of Density}\label{densec}

In this Section we list some properties of  the arithmetic density
function $\mathcal D(\l,\k_1,\k_1)$. 

\begin{definition}[Density] \label{densitydef}
\begin{enumerate}
\item Let $\k_1\le \k_2\le \l$ be cardinals. A set $\mathcal D\su
[\l]^{\k_1}$ is \emph{dense in $[\l]^{\k_2}$} if for every
  $X\in [\l]^{\k_2}$ there is $Y\in \mathcal D$ such that $Y\su X$.
\item For $\k_1\le\k_2\le
 \l$ the \emph{$(\k_1,\k_2)$-density of $\l$}, denoted by $\mathcal
 D(\theta,\k_1,\k_2)$,  is  the least
 cardinality of a set $\mathcal D\su [\theta]^{\k_1}$ which is dense
 in $[\l]^{\k_2}$.
\item For $\k\le$ the \emph{$\k$-density of $\l$}, denoted by
  $\mathcal D(\l,\k)$,  is  $\mathcal
  D(\l,\k,\k)$. 
\end{enumerate}
\end{definition} 

\begin{claim}[Basic properties of density]\label{densitybasic}
Suppose $\k_1\le \k_2\le \l$.
\begin{enumerate}
\item If $X\su Y$ and $\mathcal D\su [Y]^{\k_1}$ is dense in $[Y]^{\k_2}$, then $\mathcal D\cap
[X]^{\k_1} $ is dense in $[X]^{\k_2}$. 
\item (Monotonicity)  $\mathcal D(\l,\k_1,\k_2)$ is increasing in $\l$
  and if $\k'_1\le \k_1\le \k_2\le \k'_2\le \l$ then $\mathcal
  D(\l,\k'_1,\k'_2)\le \mathcal D(\l,\k_1,\k_2)$.
\item 
Suppose $\lng X_i:i<\theta\rng$ is
  $\su$-increasing, $X=\bigcup_{i<\theta}X_i$, $\k_1\le \k_2\le |X_0|$
  and $\mathcal D_i\su [X_i]^{k_1}$ is dense in
  $[X_i]^{\k_2}$. Then $\mathcal D:=\bigcup_{i<\theta}\mathcal D_i$ is
  dense in $[X]^{\k_3}$ for all $\k_3>\k_2$ such that $\k_3 \le |X|$,
  and if $\cf\theta\not=\cf\k_2$ then $\mathcal D$ is dense also in
  $[X]^{\k_2}$ .
\item(Continuity at limits) If $\l=\lng
  \l_i:i<\theta\rng$ is increasing with limit $\l$,
  and $\mathcal D(\l_i,\k_1,\k_2)\le \l$, then
  $\mathcal D(\l,\k_1,\k_3)=\l$ for all $\k_3>\k_2$  such that $\k_3 \le |X|$, and if
  $\cf\l=\cf\k_2$ then also $\mathcal D(\l,\k_1,\k_2)=\l$.
\item If $\cf\mu=\cf\k<\mu$ then $\mathcal D(\mu,\k)>\mu$.
\item If $\mu$ is a strong limit cardinal, then $\mathcal
  D(\mu,\k)=\mu$ for all $\k<\mu$ such that $\cf\k\not=\cf\mu$. 
\item $\mathcal D(\l,\k)$ is \emph{not} increasing in $\k$ and the inequality in (2) can be strict.
\end{enumerate}
\end{claim}

\begin{proof} Item (1) is follows from the definition of density. The
  monotonicity in $\l$ in (2) follows from (1). 

  Assume $\k'_1\le \k_1\le \k_2\le \k'_2\le \l$,  let $\mathcal D\su
  [\l]^{\k_1}$ be an arbitrary dense set in $[\l]^{\k_2}$ and for each
  $X\in \mathcal D$ fix $X'\in \mathcal [X]^{\k'_1}$. Let $\mathcal
  D'=\{X': X\in \mathcal D\}$. Given any $Y\in [\l]^{\k'_2}$ there is
  some $X\in \mathcal D$ such that $X\su Y$, hence $X'\su Y$. Clearly,
  $|\mathcal D'|\le |\mathcal D|$. Thus, $\mathcal D(\l,\k'_1,\k'_2)\le
  \mathcal D(\k_1,\k_2)$.

  To prove (3) assume first that $\cf\theta\not=\cf\k_2$ and let $Z\in
  [X]^{\k_2}$ be arbitrary. Since $\cf\k_2\not=\cf\theta$ there is
  some $i<\theta$ such that $|Z\cap X_i|=\k_2$, and thus there is some
  $Y\in \mathcal D_i\su \mathcal D$ such that $Y\su Z$. This
  establishes the density of $\mathcal D$ in $[X]^{\k_2}$; as
  $[X]^{\k_2}$ is dense in $[X]^{\k_3}$ for all $\l\ge\k_3\ge\k_2$,
  the density of $\mathcal D$ in $[X]^{\k_3}$ follows. Now it remains
  to prove the density of $\mathcal D$ in $[X]^{k_3}$ for
  $\l\ge\k_3>\k_2$ when $\cf\theta=\cf\k_2$.  Given $Z\in [X]^{\k_3}$
  it suffices to show that $|Z\cap X_i|\ge \k_2$ for some
  $i<\theta$. As $\k_3>\k_2$ and $\cf\k_2=\cf\theta$ actually more is true: there is
  $i<\theta$ such that $|Z\cap X_i|>\k_2$ (or else
  $|Z|\le\cf\k_2\times \k_2=\k_2$).

Item (4) follows
  from (3).

  The inequality (5) follows from the standard diagonalization
  argument in the proof of the Zermelo-K\" onig lemma which, in fact,
  proves that there exists an almost disjoint $\mathcal F\su [\mu]^\k$
  of cardinality $>\mu$. This implies that $\mathcal D(\mu,\k)>\mu$.

  For (6), let $\mu$ be a strong limit and let $\k<\mu$ be
  arbitrary. There is an unbounded set of cardinals $\l<\mu$ below $\mu$ which
  satisfy $\l^\k=\l$ hence $\mathcal D(\mu,\k)=\mu$ follows from (4)
  if $\cf\mu\not=\cf\k$. 

  To see (7) let $\nu$ be an arbitrary cardinal.  By (5), 
  $\mathcal D(\beth_\om(\nu),\aleph_0)>\beth_\om(\nu)$ while by
  (6) $\mathcal D(\beth_\om(\nu),\aleph_1)=\beth_\om(\nu)$.
  Similarly, $\mathcal
  D(\beth_\om(\nu),\aleph_0)>\beth_\om(\nu)=\mathcal
  D(\beth_\om(\nu),\aleph_0,\aleph_1)=\mathcal
  D(\beth_\om(\nu),\aleph_1)$ and $\mathcal
  D(\beth_{\om_1}(\nu),\aleph_0,\aleph_1)=\beth_{\om_1}(\nu)<\mathcal
  D(\beth_{\om_1}(\nu),\aleph_1)$, so (7) is proved. 
\end{proof}

The next claim provides a simple sufficient condition for the validity
of the equation $\mathcal
D(\l,\k_1,\k_2)=\l$ for all $\l$ in an end-segment of the
cardinals. 

\begin{claim}\label{doubledense-gen}
Suppose $\theta>\k_2>\k_1\ge \aleph_0$, $\mathcal D(\theta,\k_1,\k_2)=\theta$ and for every
$\l\ge \theta$, if $\cf\l=\cf\k_2$ there exists some $\nu$
such that $\k_1\le \nu<\k_2$ and 
$\l=\sup\{\d:\d<\l \wedge \mathcal D(\d,\k_1,\nu)\le \l\}$. Then
$\mathcal D(\l,\k_1,\k_2)=\l$ for every $\l\ge \theta$.
\end{claim}

\begin{proof}
  By induction on $\l\ge \theta$. For $\l=\theta$ the equality
  $\mathcal D(\l,\k_1,\k_2)=\l$ is assumed, and if $\l>\theta$ satisfies
  $\cf\l\not=\cf\k_2$ then $\mathcal D(\l,\k_1,\k_2)=\l$ by
  (4). Assume then that $\cf\l=\cf\k_2$. By the assumption, there
  is $\k_1\le \nu<\k_2$ and   $\lng
\l_i:i<\cf\k_2\rng$, an increasing sequence of cardinals with limit
$\l$ such that $\mathcal D(\l_i,\k_1,\nu,)\le \l$. Now
$\mathcal D(\l,\k_1,\k_2)=\l$ by (4).
\end{proof}

\begin{corollary} \begin{enumerate}
  \item Let $\nu$ be an infinite cardinal, $\k=\cf\k>\nu$ and assume
    that $\theta\ge \k$ satisfies such that $\mathcal
    D(\theta,\nu,\k)=\theta$ and for every $\mu>\theta$ all limit
    cardinals in the interval $[\mu,\mu^\nu)$ are of cofinality
    strictly smaller than $\k$. Then $\mathcal D(\l,\nu,\k)=\l$ for
    all $\l\ge \theta$.
  \item The SCH implies that for every infinite $\nu$ it holds that
    $\mathcal D(\l,\nu,\nu^+)=\l$ for all $\l\ge 2^{\nu}$.
\end{enumerate}
\end{corollary} 

\begin{proof}If $\l\ge\theta$ and $\cf\l=\k$ then $\theta^\mu<\l$ for all
  $\theta<\l$, so $\l=\sup\{\theta^\nu: \theta<\l\}$. If
  $\d=\theta^\nu$ then $\d^\nu=\d$ and trivially
  $\mathcal D(\d,\nu)=\d$. Thus (1) follows from Claim
  \ref{doubledense-gen}.  Item (2) follows from (1) since by the SCH
  for every cardinal $\l\ge 2^{\nu}$ it holds that
  $\l^{\nu}\in \{\l,\l^+\}$. 
\end{proof}

By the Corollary above, the existence of a bound on the gaps between
$\mu$ and $\mu^\k$ in some end segment of the cardinals suffices to
cover the end segment by a single equation of the form $\mathcal
D(x,\nu,\k)=x$.  However, models of set theory with arbitrarily large
gaps between $\mu$ and $\mu^{\aleph_0}$ can be built from a proper
class of suitable large cardinals by the methods in \cite{gm}.

The next Lemma and its Corollary \ref{densityalln} below, hold in
every model of ZFC.

\begin{definition} [Shelah's revised power] \label{rp} Let $\theta\le
 \l$ be cardinals. A set $\mathcal D\su [\l]^\theta$ is \emph{weakly
   covering} if for every $X\in [\l]^\theta$ there exists $\mathcal
 Y\in [\mathcal D]^{<\theta}$ such that $X\su \bigcup \mathcal Y$.

 \emph{Shelah's revised $\theta$-power of $\l$}, denoted
 $\l^{[\theta]}$, is the least cardinality of a weakly covering
 $\mathcal D\su [\l]^\theta$ .
\end{definition}

\begin{theorem}[Shelah's Revised GCH Theorem \cite{sh:460}] For
 every strong limit  cardinal $\mu$, for all $\l\ge
 \mu$, for every sufficiently large regular
 $\theta<\mu$,
\[\l^{[\theta]}=\l.\]
\end{theorem}

\begin{lemma} \label{density} Let $\mu$ be a strong limit
  cardinal. For every $\l\ge\mu$ there is some $\theta(\l)<\mu$ such
  that for every $\theta<\mu$ with $\cf\theta>\theta(\l)$ it holds
  that $\mathcal D(\l,\theta)=\l$.
\end{lemma}

\begin{proof}
  Let $\l\ge \mu$ be given and let $\theta(\l)<\mu$ be fixed
  by Shelah's Revised GCH theorem such
  that for every regular $\theta \in (\theta(\l),\mu)$ it holds that
  $\l^{[\theta]}=\l$.

  Assume first that $\theta\in (\theta(\l),\mu)$ is regular. Let
  $\mathcal D\su [\l]^\theta$ witness $\l^{[\theta]}=\l$ and, as
  $2^\theta<\l$, assume that $X\in \mathcal D\imply [X]^\theta\su
  \mathcal D$. If $Z\in [\l]^\theta$ is arbitrary, fix $\mathcal Y\in
  [\mathcal D]^{<\theta}$ such that $Z\su \bigcup \mathcal Y$. As
  $\theta$ is regular, there is $X\in \mathcal Y$ such that $Y:=X\cap
  Z$ has cardinality $\theta$. Since $Y\in \mathcal D$ and $Y\su Z$,
  we have shown that $\mathcal D$ is dense in $[\l]^\theta$.

Assume now that $\theta\in (\theta(\l),\mu)$ is singular and satisfies
$\cf\theta>\theta(\l)$. Write $\theta=\sum_{i<\cf\theta}\theta_i$ with
$\theta_i=\cf\theta_i>\theta(\l)$ for each $i$. Next fix dense
$\mathcal D_i=\{X^i_\a:\a<\l\}\su [\l]^{\theta_i}$ and dense $\mathcal
T'\su [\cf\theta\times \l]^{\cf\theta}$ with $|\mathcal T'|=\l$. The
set $\mathcal T:=\{Y\in \mathcal T: i<\cf\theta \imply |Y\cap (\{i\}\times
\l)|\le 1\}$, namely, all members of $\mathcal T'$ which are partial
functions from $\cf\theta$ to $\l$ with domain cofinal in
$\cf\theta$, has cardinality $\l$ and for every $f:\cf\theta \to
\l$ there exists $Y\in \mathcal T$ such that $Y\su f$.

Let 
\[\mathcal D=\bigl\{\bigcup\{X^i_\a: \lng i,\a\rng\in Y\}:
Y\in\mathcal  T\bigr\}.\]

Clearly, $\mathcal D\su [\l]^\theta$ and $|\mathcal D|=\l$.

Given any $Z\in [\l]^\theta$, define $f(i)=\min\{\a<\l: X^i_\a\su Z\}$
for $i<\cf\theta$. Now $\bigcup_{i<\cf\theta} X^i_{f(i)}\in
[Z]^\theta$. There is some $Y\in \mathcal T$ such that $Y\su f$. The
set $\bigcup\{X^i_\a: \lng i, \a\rng\in Y\}$ therefore  belongs to
$[Z]^\theta\cap \mathcal D$. 
\end{proof}

Now the asymptotic  equations  concerning  density follow: 

\begin{corollary}\label{densityalln}
For every infinite cardinal $\nu$ and $\rho\ge \beth_\om(\nu)$, for all
but finitely many $n<\om$ it holds that 
\begin{equation} \label{maindensityrelation}
\mathcal D(\rho,\beth_n(\nu))=\rho,
\end{equation}
and 
\begin{equation}\label{seconddensityrelation}  \mathcal D(\rho,\nu,\beth_\om(\nu))=\rho.
\end{equation}
\end{corollary}

\begin{proof}For all $n$, $\cf\beth_{n+1}(\nu)>\beth_n(\nu)$, so
  $\cf\beth_n(\nu)$ converges to $\beth_\om(\nu)$, hence
  (\ref{maindensityrelation}) follows from Lemma \ref{density}. To
  prove (\ref{seconddensityrelation}) fix, for a given
  $\rho\ge\beth_\om(\nu)$, some $n$ such that $\mathcal
  D(\rho,\beth_n(\nu))=\rho$ and as $\nu\le
  \beth_n(\nu)<\beth_\om(\nu)$, by monotonicity $\mathcal
  D(\rho,\nu,\beth_\om(\nu))\le \mathcal D(\rho,\beth_n(\nu))$.
\end{proof}

The consistency of $\mathcal D(\rho,\nu,\beth_n(\n))>\rho$ for
$\beth_n(\nu)\le \rho<\beth_\om(\nu)$ would imply that
$\beth_\om(\nu)$ cannot be relaxed to $\beth_n(\nu)$ in
(\ref{seconddensityrelation}) above. See Section \ref{swh} below for a
discussion of this --- at the moment unknown --- consistency.

Shelah pointed out to me the next lemma, which will be used in the
comparison theorem in the last theorem of Section \ref{families}.

\begin{lemma}\label{continuity} 
  Suppose $\mu$ is an strong limit cardinal. Let $\chi$ be a sufficiently large
  regular cardinal and $\d < \chi$ is a limit ordinal.  Suppose a sequence
  $\lng M_i:i\le\d\rng$ of elementary submodels of
  $(H(\chi),\in,\dots)$ satisfies:
\begin{enumerate} 
\item $M_i\su
   M_{i+1}$ and $M_j=\bigcup_{i<j}M_i$ when $j\le\d$ is limit;
 \item $\lng M_j:j\le i\rng \in M_{i+1}$ for all $i<\d$;
\item   $\mu\su M_0$ and $|M_i|\su M_i$ for all $i\le \d$. 
\end{enumerate} 
Then there exists $\k(*)<\mu$ such that
 $M_\d\cap [M_\d]^{\k}$ is dense in $[M_\d]^{\k}$ for
 all $\k$ such that $\k(*)\le \cf\k\le \k<\mu$. 
\end{lemma}

\begin{proof}
  Denote $\l=|M_\d|$ and for $i<\d$ denote $\l_i=|M_i|$.  Let
  $\k(*)<\mu$ be such that $\mathcal D(\l,\k)=\l$ for all $\k<\mu$
  with $\cf\k\ge\k(*)$. By increasing $\k(*)$ if necessary, we assume
  that $\cf \d\notin [\k(*),\mu)$. Given $i<\d$ and $\k$ such that
  $\k(*)\le \cf\k\le \k<\mu$, there exists a dense $\mathcal D^\k_i\su
  [M_{i}]^{\k}$ in $M_{i+1}$ of size $\mathcal D(\l_i,\k)$ by
  elementarity. As $\l_{i}\le \l$, it holds by \ref{densitybasic}(2)
  that $|\mathcal D^\k_i|\le \l$ for all $i<\d$. Since $\mathcal
  D^\k_i\in M_\d$ and $|\mathcal D^\k_i|\le \l\su M_\d$, it follows
  that $\mathcal D^\k_i\su M_\d$. Now $\mathcal D^\k:=\bigcup_{i<\d}
  \mathcal D^\k_i\su M_\d$ for all $\k<\mu$ with $\cf\k\ge\k(*)$ and
  has cardinality $\l$. As $\cf \d\not=\cf\k$ for such $\k$, the set
  $\mathcal D^\k\su M_\d$ is dense in $[M_\d]^{\k}$.
\end{proof}

\section{Filtrations  with respect to anti-monotone set functions}
\label{filtsec}

Every infinite subset of a structure with countably many finite-place
operations is contained in a subset of the same cardinality which is
closed under all operations by the downward Skolem-L\"owenheim
theorem. Also, the union of any increasing chain of closed sets is
closed. Consequently, every uncountable structure with countably many
finitary operations is \emph{filtrable}, that is, presentable as an
increasing and continuous union of substructures of smaller
cardinality. 

If infinitary operations are admitted, both facts above are no longer
true since the cardinality of the closure of a subset of cardinality
$\l$ under a $\k$-place operation depends on the valued of the
exponent $\l^\k$ which is undecidable an is periodically larger than
$\l$ --- e.g. for $\l$ with countable cofinality --- in every model of
set theory, including models of the GCH.

 The main
theorem of this section asserts that for anti-monotone set-functions
a version of the Skolem-L\"owenheim theorem and
filtrations to closed sets  exist without appealing to additional axioms. 

\begin{definition} 
\begin{enumerate}
\item 
A \emph{filtration} of
an infinite set  $V$ is a sequence of
  sets $\lng D_\a:\a<\k\rng$ for some cardinal $\k$ which satisfies: 
\begin{enumerate} 
\item $|D_\a|<|V|$. 
\item $\a<\beta<\k\imply D_\a\su D_\b$
\item If $\a<\k$ is limit then $D_\a=\bigcup_{\beta<\a}D_\beta$. 
\item $V=\bigcup_{\a<\k} D_\k$.
\end{enumerate}
Condition (c) in the definition is the \emph{continuity} of $\lng D_\a:\a<\k\rng$.
\item Suppose $V$ is an infinite set and $\mathcal C\su \mathcal
  P(V)$. A $\mathcal C$-filtration of $V$ is a filtration $\lng
  D_\a:\a<\k\rng$ such that $D_\a\in \mathcal C$ for all $\a<\k$. 
\item We say that $V$ is \emph{$\mathcal C$-filtrable}, for $\mathcal
  C\su \mathcal P(V)$, if there exists a $\mathcal
C$-filtration of $V$.
\end{enumerate}
\end{definition}

\begin{theorem} \label{Si}
Suppose  $\mathcal C\su \mathcal P(V)$,  
  $I\not=\emptyset$
  is  countable and
  $\rho <|V|$ a cardinal. 
  Suppose $\mathcal S_i\su \mathcal C$ for $i\in I$ and:
\begin{enumerate}
\item For every $i\in I$, the union of every chain of sets from $S_i$ belongs
  to $\mathcal C$. 
\item For every $\rho\le \theta<|V|$ there exists $i\in
  I$ such that for every set $A\in [V]^{\theta}$ there 
  is $D\in \mathcal S_i$ such that  $A\su D\in [V]^\theta$. 
\end{enumerate}
Then $V$ is $\mathcal C$-filtrable. 
\end{theorem}

\begin{proof}
  Let $\l:=|V|$.

  Assume first that $\cf \l>\aleph_0$.  Fix an increasing and
  continuous chain of sets $\lng A_\a:\a<\cf\l\rng$ such that
  $V=\bigcup_{\a<\cf\l} A_\a$ and $|A_\a|<\l$ for each $\a<\cf\l$ and
  $|A_0|\ge \rho$. 
  Denote
  $\theta_\a:=|A_\a|<\l$.

As $\cf\l>\aleph_0$
  is regular and $I$ is countable, 
  we may assume, by passing to a subsequence, that for some fixed
  $i\in I$  it holds that  for every $\a<\cf\l$ and every $B\in
  [V]^{\theta_{\a+1}}$ there is $D\in \mathcal S_i$ such that $B\su D$
  and $|D|=|B|$. 

  Define inductively $D_\a$ for $0<\a<\cf\l$. For limit $\a<\cf\l$ let
  $D_\a=\bigcup_{\beta<\a}D_\beta$. For 
  $\a=\beta+1$ choose 
  $D_\a\in \mathcal S_i$ which contains $A_\a\cup\bigcup_{\b<\a}
  D_\beta$ and is of cardinality $|A_\a\cup\bigcup_{\b<\a}
  D_\beta|$.

  $V=\bigcup_{\a<\cf\l}$ since $A_\a\su D_\a$ for all $\a<\cf\l$, 
  and as $D_{\a+1}\in \mathcal S_i$, each
  $D_{\a+1}$ belongs to $\mathcal C$. 
  By (1) $D_\a\in \mathcal C$ for
  limit $\a<\cf\l$ as well. Thus $\lng D_\a:0<\a<\cf\l\rng$ is a
  $\mathcal C$-filtration. 

If $\cf\l=\aleph_0$ fix an increasing union  $\l=\bigcup_{n<\om} A_n$
with $|A_n|<\l$ and  $|A_0|\ge \rho$ and let $D_{n+1}=K(A_{n+1}\cup D_n)$. By (2),
$|D_{n+1}|=|A_{n+1}|$ hence $\lng D_{n+1}:n<\om\rng$ is a $\mathcal
C$-filtration. 
\end{proof}

\begin{definition}
Let $V$ be a nonempty set. 
\begin{enumerate}
\item A \emph{notion of closure} over $V$ is a family of sets
  $\mathcal C\su \mathcal P(V)$ that satisfies: 
\begin{enumerate}
\item $\emptyset,V\in \mathcal C$. 
\item $D_1,D_2\in \mathcal C\imply D_1\cup D_2\in \mathcal C$
\item $\bigcap X\in \mathcal C$ for all $X\su \mathcal C$. 
\end{enumerate}
\item A function $K:\mathcal P(V)\to \mathcal P(V)$ 
is a \emph{closure operator} if it satisfies 
Kuratowski's {closure axioms}:
\begin{enumerate}
\item $K(\emptyset)=\emptyset$;
\item $A\su K(A)$ for $A\su V$;
\item $K(K(A))=K(A)$ for $A\su V$;
\item $K(A\cup B)=K(A)\cup K(B)$ for $A,B\su V$. 
\end{enumerate}
\item A \emph{notion of semi-closure} over $V$ is a family $\mathcal
  C\su \mathcal P(V)$ which satisfies $V\in \mathcal C$ and $\bigcap
  X\in \mathcal C$ for all $X\su \mathcal C$. 
\item  A function $K:\mathcal P(V)\to \mathcal P(V)$ 
is a \emph{semi-closure operator} if it satisfies Kuratowski's axioms
(b) and (c) and the following implication of (d):
\begin{enumerate}
\item[(d')] $A_1\su A_2\su X\imply K(A_1)\su K(A_2)$,
\end{enumerate}
\item Given a notion of semi-closure $\mathcal C\su \mathcal P(V)$ let
  $K_{\mathcal C}(A)=\bigcap \{ D: A\su D\in \mathcal C$. Conversely,
  given a semi-closure operator $K:\mathcal P(V)\to \mathcal P(V)$ let
  $\mathcal C_K=\{D: D=K(D)\su V\}$. 
\end{enumerate}
 \end{definition}

For every semi-closure operator $K$ on a set $V$ it holds that
$K=K_{\mathcal C_K}$ and for every notion of semi-closure $\mathcal C$
it holds that $\mathcal C=\mathcal C_{K_{\mathcal C}}$. If $K$ is a
closure operator then $\mathcal C_k$ is a notion of closure and if
$\mathcal C$ is a notion of closure then $K_{\mathcal C}$ is a closure
operator. Thus,  notions of [semi]-closure and 
[semi]-closure operators are
 interchangeable.

\begin{definition}
  Suppose that $F:\mathcal P(V)\to \mathcal P(V)$ and $\k$ is a
  cardinal. Denote by $F_\k$ the restriction $F\rest [X]^\k$ to sets
  of size $\k$.  The
  subsets which are closed under $F_\k$ form the following notion of
  semi-closure $\mathcal C_{F,\k}=\{D: D\su V \wedge (Y\in
  [D]^\k\imply F(Y)\su D)\}$.  Let $K_{F,\k}$ denote the semi-closure
  operator corresponding to $\mathcal C_{F,\k}$ and let us refer to
  sets $D\in \mathcal C_{F,\k}$ as \emph{$F_\k$-closed} sets.
\end{definition}

\begin{definition} A function $F:\mathcal P(X)\to \mathcal P(X)$ is
  \emph{anti-monotone} if $A\su B\su X \imply F(B)\su F(A)$.
\end{definition}
\begin{claim} \label{cofdif}
  Suppose $F:\mathcal P(V)\to \mathcal P(V)$ is anti-monotone and $\k$
  is an infinite cardinal. Then:
\begin{enumerate}
\item $K_{F,\k}$ is a closure operator.
\item An increasing union $\bigcup_{\a<\theta}D_\a$ of
  $K_{F,\k}$-closed sets is $K_{F,\k}$-closed if $\cf\theta\not=\cf \k$
\end{enumerate} 
\end{claim}

\begin{proof}
We verify
  axiom (4). If $X\in [A\cup B]^\k$ then, since $\k$ is infinite, $|X\cap
  A|=\k$ or $|X\cap B|=\k$, so by anti-monotonicity $F(X)\su K(A)$ or
  $F(X)\su K(B)$.

Assume now that $D=\bigcup_{\a<\theta} D_i$ is an increasing union of
$K_{F,\k}$-closed sets and $\cf\theta\not=\cf\k$. Let $X\in [D]^\k$.
There exists some $\a<\theta$ such that $X\cap D_\a\in [D_\a]^\k$.
<Since $D_\a$ is $K_{F,\k}$-closed, $F(X\cap D_\a)\su D_\a$ and by
anti-monotonicity $F(X)\su F(X\cap D_\a)$. 
\end{proof}

\begin{lemma}\label{stabdense} Suppose $V$ is a set and  $F:\mathcal
  P(V) \to \mathcal P(V)$ is anti-monotone. Suppose $\k_1\le \k_2<\rho$
  are cardinals and $|F(X)|\le \rho$ for all $X\in [V]^{\k_1}$.  Let
  $K=K_{F,\k_2}$. If $\theta\ge\rho$ satisfies $\mathcal
  D(\theta,\k_1,\k_2)=\theta$ then $|K(A)|=\theta$ for all $A\in
  [V]^\theta$.
\end{lemma}

\begin{proof}
  Given $A\in [V]^\theta$ put $A_0=A$ and for $\zeta\le\k_2^+$ define
  $A_\zeta\in [V]^\theta$ inductively. For limit $\zeta\le \k_2^+$ let
  $A_\zeta=\bigcup_{\xi<\zeta} A_\xi$, which is of cardinality
  $\theta$. To define $A_{\zeta+1}$ for $\zeta<\k_2^+$ fix, by the
  assumption $\mathcal D(\theta,\k_1,\k_2)=\theta$, a set $\mathcal
  D_\zeta\su [A_\zeta]^{\k_1}$ of cardinality $\theta$ with the
  property that for all $Y\in [A_\zeta]^{\k_2}$ there exists $X\in
  \mathcal D_\zeta$ such that $X\su Y$ and let
  $A_{\zeta+1}=A_\zeta\cup \bigcup\{F(X): X\in \mathcal D_\zeta\}$, which is
  of cardinality $\theta$ since $\rho\le \theta$.

For every  $X\in [A_{\k_2^+}]^{\k_2}$ there is some $\zeta<\k_2^+$
such that $Y\in [A_\zeta]^{\k_2}$ and hence there is some $X\in
\mathcal D_\zeta$ such that $X\su Y$. As $F(Y)\su F(X)\su
A_{\zeta+1}$, it holds that $F(Y)\su A_{\k_2^+}$, so $A_{\k_2^+}$ is
$K$-closed.
\end{proof} 



\begin{theorem}[Asymptotic filtrations for anti-monotone set functions]\label{filtrable} 
  Let $\nu$ be an infinite cardinal and denote $\mu:=\beth_\om(\nu)$.
  Let $V$ be a set of cardinality $|V|>\mu$ and suppose $F:\mathcal
  P(V)\to \mathcal P(V)$ is an anti-monotone function.  If there
  exists a cardinal $\rho\ge\mu$ such that $\rho<|V|$ and $|F(Y)|\le
  \rho$ for every $Y\in [X]^\nu$ then $V$ is $K_{F,\mu}$-filtrable.
\end{theorem}

\begin{proof} Let $\mathcal C$ be the family of all $K_{F,\mu}$-closed
  subsets of $V$. Let $\k_n$ denote $\beth_n(\nu)$ and let
  $K_n=K_{F,\k_n}$.   
Let $\mathcal S_n$ be the collection of all $K_n$-closed
subsets of $V$. We check that conditions (1) and (2) in Theorem
\ref{Si} are
  satisfied for $\mathcal C$ and $\{\mathcal S_n: n<\om\}$.

  As $\k_n<\mu$ and $F$ is anti-monotone,  it holds that  $\mathcal S_n\su \mathcal S_{n+1}\su
  \mathcal C$ for all $n$. The union of any chain of sets
  from $\mathcal
  S_n$ belongs to $\mathcal S_{n+1}\su \mathcal C$ by Lemma
  \ref{cofdif} (2).

  To verify condition (2) let $\theta\ge \rho$ be given. By Lemma
  \ref{density}, there exists $m(\theta)<\om$ such
  that $D(\theta,\k_n)=\theta$ for all $n\ge m(\theta)$. By Lemma \ref{stabdense},
  $|K_n(A)|=\theta$ for $A\in [V]^\theta$ for all $n\ge
  m(\theta)$.

Now 
  $V$ is $K_{F,\nu}$-filtrable by Theorem \ref{Si}. 
\end{proof}

We remark that each $\mathcal S_n$ is in fact the collection of closed
sets with respect to a closure operator    and that
$m<n\imply \mathcal S_m\su \mathcal S_n$. Neither of these two
properties was used in proving Theorem \ref{filtrable}.

The following lemma relates filtrations to elementary chains of models:

\begin{lemma}Suppose $F$ is anti-monotone, $|F(X)|\le \rho$ for all $X\in [V]^\nu$
and $\rho\ge \mu=\beth_\om(\nu)$. Then for every chain $\lng
M_i:i<\l\rng$ of elementary submodels of a sufficiently large
$(H(\chi),\in)$ which satisfies $\mu\su M_0$, $|M_i|\su M_i|$,  $V,F\in M_0$  and $\lng
M_j:j\le i\rng\in  M_{i+1}$ the set
$D_i:=M_i\cap V$ is $F_\mu$ closed for every limit $i<\l$.
If $|M_i|<\l:=|V|$ then\
$\lng D_i: i<\l \text{ is limit}\rng$ is a $\mu$-filtration of $V$.
\end{lemma}

\begin{proof}
If $M\prec (H(\chi),\in,\dots)$ and $V,F\in M$, $\rho\su M$, then for every $X\in
M\cap [V]^\k$ for $\nu\le \k<\nu$ the set $F(X)$
belongs to $M$ and $|F(X)|\le \rho$, hence $F(X)\su M$. Thus, if $M\cap
[V]^\k$ is dense in $[V]^\k$, the set $V\cap M$ is $\k$-closed. By \ref{continuity}
this is the case for $M_i$ for all limit $i<\l$. 

If indeed $|M_i|<\l$ for all $i$ the $|D_i|<\l$. The fact that
$\bigcup_{i<\l}D_i=V$ follows from $V\in M_0$ and $\l\su\bigcup_{i<\l} M_i$.
\end{proof}

\subsection{The need for $\rho\ge\beth_\om(\nu)$}\label{swh}
Is the restriction $\rho\ge\beth_\om(\nu)$ optimal in Theorem
\ref{filtrable}?  This is not known at the moment. As a result of a
discussion \cite{personal} with Shelah, something can
be said about it: improving it will be at least as hard as proving one
of the following versions of Shelah's Weak Hypothesis, indexed by $n>0$:
\begin{enumerate}
\item[(SWH$_n$)] There are  no infinite $\nu$ and $\rho$ 
such that $\beth_n(\nu)<\rho<\beth_\om(\nu)$ and $\mathcal
F=\{A_\a:\a<\rho^+\}\su [\rho]^\rho$ satisfies $|A_\a\cap A_\b|<\nu$
for $\a<\b<\rho^+$.
\end{enumerate}

Shelah's Weak Hypothesis is a dynamic statement whose evolving
contents is the weakest unproved pcf-theoretic statement. It is
formulated in terms of patterns in Shelah's function pp.  See
\cite{sh:410} for the translation of SWH from pcf to almost disjoint
families. The negations of first two versions of the SWH have been
shown consistent. Gitik \cite{gm} showed the consistency of an
$\omega$ sequence of cardinals with pp of each exceeding the limit of
the sequence. SWH was then revised to the statement that there was no
such $\om_1$ sequence. Recently, Gitik \cite{gitikse} proved the
consistency of the existence of an $\om_1$ sequence of cardinals with
the pp of each exceeding the limit of the sequence.

Shelah's believes that the negations of the versions above
will eventually be proved consistent \cite{personal}.

Suppose that $\beth_n(\nu)<\rho<\beth_\om(\nu)$ and $\mathcal
F=\{A_\a:\a<\rho^+\}\su [\rho]^\rho$ satisfies $|A_\a\cap A_\b|<\nu$
for $\a<\b<\rho^+$. Let $V=\rho\dot\cup \mathcal F$ and let
$F(X)=\{\a:X\cap \rho\su A_\a\}\cup \bigcap _{A\in X\cap \mathcal F}
A$. This is an anti-monotone function and for every $X$ with $|X|\ge
\nu^+$ it holds that $|F(X)|<\nu$. Yet, there is no filtration of $V$
to $F_\nu$-closed sets, as $\rho$ would have to be contained in one of
the parts, and every $F_\nu$-closed set which contains $\rho$ is equal
to $V$.

Thus, a ZFC proof of Theorem \ref{filtrable} with $\beth_n(\nu)$ in
place of $\beth_\om(\nu)$ for some $n>0$ will imply the $n$-th version
of the Shelah Weak Hypothesis. If, however, the negations of (SWH$_n$)
are consistent for all $n>0$, then $\beth_\om$ is optimal.

\section{Splitting families of sets}\label{families}
Following \cite{eh} let us say that a family $\mathcal F$ of sets
satisfies condition $C(\l,\k)$ if the intersection of every subfamily
of $\mathcal F$ of size $\l$ is of size strictly less than $\nu$.
Miller \cite{miller} define \emph{property B} (after Felix Bernstein's
``Bernstein set'') of a family of sets $\mathcal F$ as: there exists a
set $B$ such that $B\cap A\not =\emptyset$ and $A\not\su B$ for all
$A\in \mathcal F$ and proved that for every infinite $\rho$, every
$\rho$-uniform family of sets that satisfies $C(\rho^+,n)$ for some
natural number $n$ satisfies property B. Erd\H os and Hajnal \cite{eh}
used Miller's method and a theorem of Tarski \cite{tarski} to prove
from the assumption GCH that every $\rho$-uniform family that
satisfies $C(\rho^+,\nu)$ for an infinite cardinal $\nu$ satisfies
property B if $\rho>\nu^+$.

More results with the GCH followed along this line
\cite{komclose,komcomp,hjs,hjss}.  Komj\'ath \cite{komclose} proved from
the GCH that for $\nu^+<\rho$ every $\rho$-uniform $\mathcal F$ which
satisfies $C(\rho^+,\nu)$ and is also almost disjoint is
\emph{essentially disjoint}, that is, it can be made pairwise disjoint
by removing a set of size $<\rho$ from each member of $\mathcal F$.

Komj\'ath \cite{komcomp} investigated further the property of essential
disjointness which he introduced in \cite{komclose} (under the name
``sparseness'') and proved that the array of cardinalities of pairwise
intersections in an $\aleph_0$-uniform a.d family determines whether
the family is essentially disjoint or not.

Hajnal, Juh\'asz and Shelah proved in \cite{hjs} a general theorem
which implied many of the Miller-type theorems known at the time and
derived also new combinatorial and topological consequences with it by
assuming as an additional axiom a relaxation of the GCH to a weak
version of the Singular Cardinal Hypothesis.

\begin{definition}
Let $\mathcal F$ be a family of sets. 
\begin{enumerate}
\item The \emph{universe of $\mathcal F$}  is the set $\bigcup F$,
and is denoted by $V(\mathcal F)$.
\item Given $U\su V(\mathcal F)$ let $\mathcal F(U)=\{\mathcal F\cap
  \mathcal P(U)\}=\{A:A\in \mathcal F \wedge A\su U\}$. 
\item $\mathcal F$ is $\rho$-uniform, for a cardinal $\rho$,  if $|A|=\rho$
for all $A\in \mathcal F$. 
\end{enumerate}
\end{definition}

\begin{definition}Suppose that $\mathcal F$ is a family of sets and
  $V=V(\mathcal F)$. 
\begin{enumerate}
\item A set $U\su V$ is $\k$-closed,  for a cardinal $\k$, if 
  $|A\cap U|\ge \k\imply A\su U$ for all $A\in \mathcal F$.
\item Let $F:\mathcal P(V)\to \mathcal P(V)$ be defined
  by 
\[F(X)=\bigcup\{A: X\su A\in \mathcal F\}.\]
\end{enumerate}
\end{definition}

\begin{claim} For every family $\mathcal F$ with $V=V(\mathcal F)$ and
  cardinal $\k$,
\begin{enumerate}
\item The collection $\mathcal C_\k$ of all
  $\k$-closed sets is a notion of semi-closure over $V$ and if $\k$ is
  infinite then $\mathcal D_\k$ is a notion of closure.
\item If $U\su V$ is $\k$-closed and $A\in \mathcal F$ is not
  contained in $U$ then $|A\cap U|<\k$.
\item The function $F$ defined above  is anti-monotone. 
\item For every $U\su V$, $U$  is $\k$-closed iff $U$ is
  $K_{F,\k}$-closed.
\item If $\k_1\le \k_2$ are cardinals then $\mathcal C_{\k_1}\su
  \mathcal C_{\k_2}$.
\item If $\mathcal F$ is $\rho$-uniform and $\k\le \rho$ then a set $U\su V$
  is $\k$-closed iff $\mathcal F(U)=\{A:A\in \mathcal F \wedge |A\cap
  U|\ge \k\}$.
\item Abusing notation, we say that a subfamily $\mathcal F'\su \mathcal F$ is
  $\k$-closed if $V(\mathcal F')$ is $\k$-closed and we denote by
  $K_{F,\k}(\mathcal F)$ the family $\mathcal F(K_{F,\k}(V(\mathcal
  F)))$. 
\end{enumerate}
\end{claim}

\begin{definition}[Disjointness conditions]
\begin{enumerate}
\item $\mathcal F$ is \emph{almost disjoint}  (a.d.)  if $|A\cap
  B|<\min\{|A|,|B|\}$ for distinct $A,B\in \mathcal F$. 
\item $\mathcal F$ is $\nu$-disjoint, for a cardinal $\nu$, if $|A\cap
  B|<\nu$ for any distinct $A, B\in \mathcal F$. 
\item (Miller \cite{miller}, Erd\H os and Hajnal \cite{eh})
$\mathcal F$  satisfies condition $C(\theta,\nu)$ for cardinals $\theta,\nu$ if  $|\bigcap
\mathcal A|<\nu$ for all $\mathcal A\in [\mathcal F]^\theta$ (so,
$C(2,\nu)$ is $\nu$-disjointness).
\item (Komj\'ath \cite{komclose}, Hajnal, Juh\'asz and Shelah \cite{hjs}) $\mathcal
  F$ is \emph{essentially disjoint} (e.d) if there exists an assignment of
  subsets $B(A)\in [A]^{<|A|}$ for all $A\in \mathcal A$ such that the
  family $\{A\sm B(A):A\in \mathcal F\}$ is pairwise disjoint.
\item $\mathcal F$ is $\nu$-e.d if for every $A\in \mathcal F$ there
  exists an assignment $B(A)\in [A]^{\nu}$ such that $\{A\sm B(A):A\in
  \mathcal F\}$ is pairwise disjoint.  Remark: in \cite{komclose}  the term
  \emph{sparse} is used instead of ``e.d''. 
\end{enumerate}
\end{definition}

\begin{claim}
\begin{enumerate}
\item 
Suppose $\mathcal F$ is $\rho$-uniform and satisfies $C(\rho^+,\nu)$
for some $\nu<\rho$. If $\nu\le \k \le \rho \le \theta$ and
$D(\theta,\k)=\theta$ then for every $U\in [V]^\theta$ it holds that
$| \{A\in \mathcal F: |A\cap U|\ge \k\}|\le \theta$.
\item 
  Suppose $\nu$ is an infinite cardinal, $\mu=\beth_\om(\nu)$,
  $\rho\ge \mu$ and $\mathcal F$ is a $\rho$-uniform family with
  univers $V$. If $\mathcal F$ satisfies $C(\rho^+,\nu)$ then for
  every $\theta\ge \rho$ and $\mathcal F'\su \mathcal F$, $|\mathcal
  F'|=\theta\iff |V(\mathcal F')|=\theta$. 
\end{enumerate}
\end{claim}

\begin{proof}
  To prove (1) fix a dense $\mathcal D\su [U]^\k$ of cardinality
  $\theta$. If $|A\cap U|\ge \k$ for some $A\in \mathcal F$ then there
  exists $X\in \mathcal D$ such that $X\su A$. By $C(\rho^+,\nu)$ and
  $\nu\le \k$, each $X\in \mathcal D$ is contained in no more
  than $\rho$ members of $\mathcal F$, so $| \{A\in \mathcal
  F:(\exists X\in \mathcal D)(X\su A)\}|\le \theta\times \rho=\theta$.

For (2) assume first that $|V(\mathcal F')|=\theta\ge \rho$. Let $n$ be such that
  $\mathcal D(\theta,\k_n)=\theta$ (where
  $\k_n=\beth_n(\nu)$). Clearly $\mathcal F(U)\su \{A\in \mathcal F:
  |A\cap U|\ge \k_n\}$ which, by (1), has cardinality $\le \theta$. 
The converse implication is
  trivial.
\end{proof}

 By the above lemma, for all sufficiently large $n$ it holds
that $|K_n(\mathcal F')|=|\mathcal F'|$ if $|\mathcal F'|\ge
\rho$. 

\begin{lemma} \label{famfilt}
Suppose $\nu$ is an infinite cardinal, $\mu=\beth_\om(\nu)$, $\rho\ge
 \mu$ and  $\mathcal F$ is a $\rho$-uniform family with univers $V$
 such that 
  $|V|>\rho$. If $\mathcal F$ satisfies $C(\rho^+,\nu)$
  then $V$ is $\mu$-filtrable. 
\end{lemma}

\begin{proof}The function $F$ defined above is anti-monotone, for every
  $X\in [V]^\nu$ it holds that $|F(X)|\le \rho$ by $C(\rho^+,\nu)$ and
  $\rho<|V|$. By
  Theorem \ref{filtrable}, $V$ is $K_{F,\mu}$-filtrable
hence  $\mu$-filtrable. 
\end{proof}

The next theorem extends Miller's Theorem 2 from \cite{miller} when
$\nu$ is finite, and is proved in ZFC for all cardinals $\nu$.

\begin{theorem}\label{main}Suppose $\nu$ is a cardinal and
  $\rho\ge\mu:=\beth_\om(\nu)$.  For every $\rho$-uniform family
  $\mathcal F$ that satisfies $C(\rho^+,\nu)$ there exists an
  enumeration $\mathcal F=\{A_\a:\a<|\mathcal F|\}$ and a family of sets
  $\{d_\a:\a<|\mathcal F|\}$ such that

\begin{enumerate}
\item[$(i)$] $d_\a\in [A_\a]^\rho$
 for 
  every $\a<|\mathcal F|$ and $\b<\a<|\mathcal F|\imply d_\b\cap d_\a=\emptyset$;
\item[$(ii)$] $|\{\b<\a: A_\b\cap d_\a\not =\emptyset\}|<\rho$ for  all $\a<|\mathcal F|$.
\end{enumerate}
\end{theorem}

\begin{proof}
  Let $\l=|\mathcal F|$ and we prove the theorem by induction on $\l$.
  If $|\mathcal F|\le\rho$ fix any enumeration $\mathcal
  F=\{A_\a:\a<\l\}$. Every $\rho$-uniform family of cardinality at
  most $\rho$ has a disjoint refinement, so we can fix
  $d_\a\in [A_\a]^\rho$ so that $\a<\b<\l\imply d_\b\cap
  d_\a=\emptyset$.  This settles $(i)$. Conclusion $(ii)$ holds
  trivially as $\l\le \rho$.

  Assume that $\l>\rho$. Since $\mathcal F$ satisfies $C(\rho^+,\nu)$
  it is $\mu$-filtrable by \ref{famfilt}. For finite $\nu$,
  filtrability is proved in Miller's \cite{miller} --- or, in modern
  terms, is an immediate consequence of the Skolem-L\"owenhem theorem
  ---  and does not
  require \ref{famfilt}. 

  Fix a strictly increasing and continuous chain of $\mu$-closed
  families $\{\mathcal H_i:i<\cf \l\}$ such that $|\mathcal
  H_0|\ge\rho$, $|\mathcal H_i|<\l$ for all $i<\cf\l$ and $\mathcal
  F=\bigcup\{\mathcal H_i:i<\cf \l\}$. Put $\mathcal F'_{i}=\mathcal
  H_{i}\sm \bigcup_{j<i}\mathcal H_j$ and let
  $\mathcal F_i=\mathcal F'_j$ for the least $j\ge i$ such that
  $\mathcal F'_j$ is not empty.

Now $\{\mathcal F_i:i<\cf\rho\}$ is a
partition of $\mathcal F$, $|\mathcal F_i|<\l$ for each $i$  and $A\in \mathcal F_i$ implies that 
\begin{equation}\label{back}|A\cap
\bigcup_{j<i}\mathcal F_j|<\mu. 
\end{equation}

Denote $\l_i=|\mathcal F_i|$. As $\l_i<\l$ for every $i<\cf\l$ and
$\mathcal F_i$ satisfies $C(\rho^+,\nu)$, the induction hypothesis
allows us to fix for each $i$ an enumeration $\mathcal F_i=\{A_{\lng
  i,\gamma\rng}:\g<\l_i\}$ and $\{d'_{\lng
  i,\gamma\rng}:\gamma<\l_i\}$ such that $d'_{\lng i,\gamma\rng}\in
[A_{\lng i,\g\rng}]^\rho$ for each $\g<\l_i$, $d'_{\lng i,\g\rng}\cap
d'_{\lng i,\d\rng}=\emptyset$ for all $\g<\d<\l_i$ it holds that and
for each $\g<\l_i$:
\begin{equation}\label{lessrho}
|\{\d<\g: A_{\lng i,\d\rng}\cap
d'_{\lng i,\g\rng}\not=\emptyset\}|<\rho. 
\end{equation}

For $i<\cf\l$ and $\g<\l_i$ define 
\begin{equation}\label{d}
d_{\lng i,\g\rng }:=d'_{\lng
  i,\g\rng }\sm
\bigcup_{j<i}V(\mathcal F_j).
\end{equation}
 
By (\ref{back}) and $d'_{\lng i,\g\rng}\su
A_{\lng i,\g\rng}$ it holds that $|d_{\lng i,\g\rng}|=\rho$.

Let $I=\{\lng i,\g\rng:i<\cf\l\wedge \g<\l_i\}$ be well-orderede by the
lexicographic ordering $<_{lx}$ of pairs of ordinals. As $|I|=\l$ and
each proper initial segment of $I$ has cardinality $<\l$, it holds
that $\lng I,<_{lx}\rng$ is order-isomorphic to $\l$.

Identifying $\l$ with $\lng I,<_{lx}\rng$ we now have an enumeration
$\mathcal F=\{A_{\lng i,\g\rng}: \lng i,\g\rng\in I\}$. The sets
$d_{\lng i,g\rng}\in [A_{\lng i,\g\rng}]^\rho$ have been defined in
(\ref{d}) for each $\lng i,\g\rng \in I$ and satisfy $d_{\lng
  \d,j\rng}\cap d_{\lng i,\g\rng}=\emptyset$ for $\lng
j,\d\rng\not=\lng i,\g\rng$ in $I$. This shows that the enumeration
of $\mathcal F$ and the family $\{d_{\lng i,\g\rng}:\lng i,\g\rng\in
I\}$ satisfy conclusion $(i)$ of the theorem.

To show that $(ii)$ also holds, let $\lng i,\g\rng\in I$ be arbitrary and let
\begin{equation}
X=\{\lng j,\d\rng<_{lx}\lng i,\g\rng :  A_{\lng j,\d\rng}\cap d_{\lng
  i,\g\rng}\not=\emptyset\}.
\end{equation}

If $j<i$ then $\lng j,\d\rng\notin X$ as $A_{\lng j,\d\rng} \cap d_{\lng i,\g\rng}
=\emptyset$ by (\ref{d}), hence 
\begin{equation}
X=\{  \lng
  i,\d\rng <_{lx} \lng i,\g\rng : A_{\lng
    i, \d\rng}\cap d_{\lng i,\g\rng}\not=0\}. 
\end{equation}

As $d_{\lng i,\g\rng}\su d'_{\lng i,\g\rng}$, it
  follows from (\ref{lessrho}) that $|X|<\rho$.
\end{proof} 

Theorem \ref{main} and Lemma \ref{famfilt} have the following
corollaries. 

\begin{corollary} \label{elimination}
If $\nu$ is infinite and $\rho\ge\mu:=\beth_\om(\nu)$ then for every
$\rho$-uniform family $\mathcal F$:
\begin{enumerate}
\item  If $\mathcal F$ satisfies $C(\rho^+,\nu)$ then $\mathcal F$ has a disjoint refinement.
\item  If $\mathcal F$ satisfies $C(\rho^+,\nu)$ and for a cardinal
  $\mu\le\theta\le \rho$ every
  subfamily of $\mathcal F$ of cardinality $\rho$ is $\theta$-e.d  then $\mathcal F$
  is $\theta$-e.d. 

\item  If $\mathcal F$ is $\nu$-a.d then $\mathcal F$ is e.d.

\item If $\rho$ is regular and $\mathcal F$ is a.d and satisfies
  $C(\rho^+,\nu)$ then $\mathcal F$ is e.d. 
\end{enumerate}
\end{corollary}

Corollary (1) extends to infinite $\nu$ the
case of $\rho$-uniform $\mathcal F$ in Theorem 3 of \cite{komclose}
and eliminates the GCH from a stronger form of a Theorem 6 in
\cite{eh}; (3) eliminates the additional axiom $A(\nu,\rho)$ from
Theorem 2.4 in \cite{hjs} for sufficiently large $\rho$ --- in the
notation there, $\operatorname{ED}(\nu,\rho)$ for all
$\rho\ge\beth_\om(\nu)$ in ZFC.  (4)
eliminates the GCH from Theorem 5 in \cite{komclose} for sufficiently
large $\rho$.

\begin{proof}
  (1) follows directly from the theorem. 

  To prove (2) let $|\mathcal F|=\l>\rho$ and assume, by
  $\mu$-filtrability, that $\mathcal F$ is partitioned to $\{\mathcal
  H_\a:\a<\l\}$ such that $|\mathcal H_\a|<\l$ and $|A\cap
  \bigcup_{\b<\a}V(\mathcal H_\b)|<\mu$ for all $A\in \mathcal
  H_\a$. Fix, by the induction hypothesis, $B_\a(A)\in [A]^{<\theta}$
  for all $A\in \mathcal H_\a$ such that $\{A\sm B_\a(A):A\in \mathcal
  H_\a\}$ is pairwise disjoint. For $A\in \mathcal F$ let
  $B(A)=B_\a(A)\cup (A\cap \bigcup_{\b<\a}V(\mathcal H_\b))$ for the
  unique $\a$ such that $A\in \mathcal H_\a$. As $|A\cap
  \bigcup_{\b<\a}V(\mathcal H_\b)|<\mu$ it holds that
  $|B(A)|<\theta$. Clearly, $\{A\sm B(A):A\in \mathcal F\}$ is
  pairwise disjoint.

  For (3): every $\nu$-a.d. $\rho$-uniform family satisfies $C(2,\nu)$
  hence $C(\rho^+,\nu)$. If $\mathcal F'\su \mathcal F$ has
  cardinality $\rho$ and is well ordered by $\mathcal
  F'=\{A_\a:\a<\rho\}$ let $B(A_\a)=A_\a\cap \bigcup_{\b<\a}A_\b$. The
  cardinality of $B_\a$ is strictly smaller than $\rho$ and $\{A_\a\sm
  B_\a:\a<\rho\}$ is disjoint. Now (3) follows from (2).

For (4), as $\rho$ is regular and $\mathcal F$ is a.d, every subfamily
of $\mathcal F$ of cardinality $\rho$ is e.d and now use (2).
\end{proof}

Neither of the corollaries above used conclusion $(ii)$ of Theorem
\ref{main}.  A generalization of this condition is used for the
next theorem on \emph{list conflict-free numbers}.

\begin{definition}
\begin{enumerate}

\item A coloring $c$ of $V(\mathcal F)$ for a family of sets $\mathcal
  F$ is \emph{conflict free} if for
 every $A\in \mathcal F$ there is $x\in A$ such that $c(x)\not=c(y)$
 for all $y\in A\sm \{x\}$. 
\item The \emph{conflict free number} $\chi_{CF}(\mathcal F)$ is the
  smallest cardinal $\k$ for which there exists a conflict-free
  coloring $c:V(\mathcal F)\to \k$. 
\item The
  \emph{list-conflict-free number} $\chi_{\ell
    CF}(\mathcal F)$ is the smalles cardinal $\k$ such that for every
  assignments $L$ of sets $L(v)$ to every $v\in V$ which satisfies
  $|L(v)|\ge \k$ there exists a conflict-free coloring $c$ on $V$ which
  satisfies $c(v)\in L(v)$ for all $v\in V$. 
\end{enumerate}
\end{definition} 

Clearly, every $\nu$-e.d family $\mathcal F$ has a coloring by
$\nu$-colors such that all but $<\nu$ colors in each $A\in \mathcal F$
are unique in $A$.

\begin{claim}\label{concofdisj} Let $\rho\ge\theta>0$ be cardinals  and let  $\mathcal
  F$ is a nonempty $\rho$-uniform family. Suppose there exists an
  enumeration $\mathcal F=\{A_\a:\a<\l\}$ and a family $\{d_\a:\a<\l\}$ such
  that
\begin{enumerate}
\item[$(i)$] $d_\a\in [A_\a]^\theta$ for every $\a<\l$
  and $\b<\a<\l\imply d_\b\cap d_\a=\emptyset$;
\item[$(iii)$] $|\{\b<\a: A_\b\cap d_\a\not =\emptyset\}|\le\rho$ for  all $\a<\l$.
\end{enumerate}
Then 
\begin{enumerate}
\item  
For every list-assignment $L(v)$ for $v\in V$ with $|L(v)|\ge \rho^+$ there exists a
coloring $c\in \prod_{v\in V}L(v)$ such that for every $\a<\l$ and
every $x\in D_\a$ the color $c(x)$ is unique in $A_\a$.
\item $\Licf(\mathcal F)\le \rho^+$.
\end{enumerate}
\end{claim}

\begin{proof}
Assume that $L(v)$ is given with $|L(v)|\ge \rho^+$
for all $v\in V(\mathcal F)$. For each $\a<\l$ enumerate $d_\a=\lng x_{\lng
  \a,i\rng}:i<\theta\rng$. Define $c\in \prod_{v\in V}L(v)$
arbitrarily on $V\sm \bigcup_{\a<\l}d_\a$. Next define $c(x_{\lng
  \a,i\rng})$ by induction on the lexicographic ordering on $\l\times
\theta$. Suppose $c(x_{\lng\b,j\rng})$ is defined for $\lng
\b,j\rng<_{lx}\lng \a,i\rng$. Define

\begin{equation}\label{forbid}
F(x_{\lng\a,i\rng}) =\{c(y):(\exists \b\le \a)\,[x_{\lng\a,i\rng}\in
A_\b\wedge y\in A_\b \wedge  c(y) \text{ is
  defined}]\}
\end{equation}
 and choose 
\begin{equation}\label{choice} c(x_{\lng\a,i\rng})\in
L(x_{\lng\a,i\rng})\sm F(x_{\lng\a,i\rng}).
\end{equation}

By the assumption (3), $|F(x_{\lng \a,i\rng})|\le \rho$,
so as $|L(x_{\lng \a,i\rng})|\ge \rho^+$ the set
$L(x_{\lng\a,i\rng})\sm F(x_{\lng\a,i\rng})$ is not empty and therefore
$c(x_{\lng \a,i\rng})$ can be chosen as required in (\ref{choice}).

To prove that the color $c(x_{\lng \a,i\rng})$ is unique in $A_\a$ for
every $\a<\l$ and $i<\theta$ suppose that $\a<\l$ and $i<\theta$ are
given and $z\in A_\a\sm \{x_{\lng \a,i\rng}\}$. If $c(z)$ is defined
at stage $\lng \a,i\rng$ of the inductive definition, then $c(z)\in
F(x_{\lng \a,i\rng})$ by (\ref{forbid}) and hence $c(x_{\lng
  \a,i\rng})\not=c(z)$ by (\ref{choice}).

 Otherwise, there is some $\b>\a$ and
$j<\theta$ such that $z=x_{\lng \b,j\rng}$. This means that $x_{\lng
  \b,j\rng}\in A_\a$. As $c(x_{\lng\a,i\rng})$ is defined at
stage $\lng \b,j\rng$, it follows by  (\ref{forbid}) that $c(x_{\lng
  \a,i\rng})\in F(x_{\lng
  \b,j\rng})$ and $c(z)=c(x_{\lng\b,j\rng)})\not=c(x_{\lng
  \a,i\rng})$ by (\ref{choice}). 
\end{proof}

 By combining Theorem \ref{main} with the Claim
 \ref{concofdisj} we get: 

\begin{corollary} \label{lcf} For every infinite $\nu$ and $\rho\ge
  \beth_\om(\nu)$, every $\rho$-uniform family which satisfies
  $C(\nu,\rho^+)$ satisfies $\chi_{\ell CF}(\mathcal F)\le \rho^+$. 
\end{corollary}

Hajnal, Juh\'asz, Soukup and Szentmil\'ossy \cite{hjss} proved
that a $\k$-uniform family which is $r$-a.d for finite $r$ has
countable conflict-free number. See also Komj\'ath \cite{komcf}, where
the uniformity condition is removed.

\subsection{Comparing almost disjoint families} We conclude this
section with a generalization of a theorem by Komj\'ath on comparing
families of sets.

\begin{definition}Two families of sets $\mathcal F$ and $\mathcal G$
  are \emph{similar} if there is a bijection $f:\mathcal F\to \mathcal
  G$ such that $|A\cap B|=|f(A)\cap f(B)|$ for all $A,B\in \mathcal
  F$. 
\end{definition}
Komj\'ath \cite{komcomp} addressed the following question: which
combinatorial properties of $\aleph_0$-uniform families are invariant under similarity?  Property B
and the existence of a disjoint refinement are not similarity
invariant, as one can replace $A\in \mathcal
F$ by $A\cup D(A)$ where $D(A)\cap D(B)=\emptyset$ for distinct
$A,B\in \mathcal F$ and obtain an equivalent family which has a
disjoint refinement from a family which does not satisfy property B.
However, Komj\'ath proved:

\begin{theorem}[Komj\'ath \cite{komcomp}]
Suppose $\mathcal F$ and $\mathcal G$ are $\aleph_0$-uniform, almost disjoint and
equivalent. If $\mathcal F$ is e.d. then also $\mathcal G$ is
e.d. More generally: $G$ is e.d if $\mathcal F$ is and there exists a bijection  $f:\mathcal F\to \mathcal G$ 
such that $|A\cap B|\ge |f(A)\cap f(B)|$ for all $A,B\in \mathcal F$.
\end{theorem}

The cardinal $\aleph_0$ is strong limit and regular, that is, strongly
inaccessible.  As a corollary of the next theorem, Komj\'ath's theorem
extends from $\aleph_0$ to all strongly inaccessible cardinals, but
actually more is proved.  

\begin{theorem}\label{komcomgen}
  Suppose $\mu$ is a strong limit cardinal and $\rho\ge \mu$. Suppose
  $\mathcal F=\{A_a:\a<\l\}, \mathcal G=\{B_\a:\a<\l\}$ are
  $\rho$-uniform and for every $\a<\b<\l$ it holds that
\begin{equation}\label{assumption}
|A_\a\cap A_\b|\ge |B_\a\cap B_\b|.
\end{equation}

Then if $\mathcal F$ is $<\mu$-essentially disjoint, so is $\mathcal
G$, provided that every subfamily of $\mathcal G$ of cardinality
$\rho$ is $\mu$-e.d.
\end{theorem}

\begin{proof}
  Let $C_\a\in [A_\a]^{<\mu}$ be fixed for all $\a<\l$ such that
  $(A_\a\sm C_\a)\cap (A_\b\sm C_\b)=\emptyset$ for all $\a<\b<\l$. We
  prove by induction on $\l\ge\rho$ that $\mathcal G$ is $<\mu$-e.d.
   If $\l=\rho$ then $\mathcal G$  is $\mu$-e.d by the assumption that every subfamily of
  $\mathcal G$ of cardinality $\rho$ is $\mu$-e.d.

  For singular $\l>\rho$ the conclusion follows from the induction
  hypothesis by Proposition 5 in \cite{komclose}.

Assume then that $\l>\rho$ is regular. We may assume that $V(\mathcal
F)=V(\mathcal G)\su \l$. As $\mathcal F$ is $<\mu$-a.d. it follows
that $|V(\mathcal F)|=\l$, so we assume that actually $V(\mathcal
F)=\l$.

Let $\ov M=\lng M_\a:\a<\l\rng$ be an elementary chain of models
$M_\a\prec (H(\Omega),\in,\prec)$ for a sufficiently large regular
$\Omega$ such that $\mathcal F, \mathcal G\in M_0$ and also the
function $A_\a\mapsto C_\a$ belongs to $M_0$, such that: 

\begin{enumerate}
\item $\rho\su M_0$ and $\a<\b\imply M_\a\su M_\b$.
\item $M_\a\cap \l\in \l$ and $|M_\a|\su M_\a$. 
\item $\lng M_\b:\b\le \a\rng\in M_{\a+1}$.
\end{enumerate}

Let $\d(\a)=M_\a\cap \l$.

For each ordinal $\g<\d(\a)$ there is at most one $\b<\l$ such that
$\g\in A_\b\sm C_\b$. Thus, whenever $\g\in A_\b\sm C_\b$ and
$\g<\d(\a)$ for some $\a<\l$, by elementarity $A_\b\in M_\a$ and hence
$A_\b\su \d(\a)$. If $\b\ge \d(\a)$, then, it follows that $A_\b\cap
\d(\a)\su C_\b$. 

For each $\b<\l$ let $\k(\b)=|C_\b|$. We get:
\begin{equation}
\b\ge \d(\a)\imply |A_\b\cap \d(\a)|\le \k(\b).
\end{equation}
and, since $\g<\d(\a)\imply A_\g\su \d(\a)$,
\begin{equation}\label{eq2}
\b\ge \d(\a)>\g\imply |A_\b\cap A_\g|\le \k(\b).
\end{equation}

By  assumption (\ref{assumption}), condition (\ref{eq2}) can be copied over to
$\mathcal G$, that is:

\begin{equation}\label{eq4}
\b\ge \d(\a)>\g\imply |B_\b\cap B_\g|\le \k(\b).
\end{equation}

\begin{claim} 
  There exists a closed unbounded $E\su \{\d(\a):\a<\l\}$ such that
  for every $\d\in E$ and $\b\ge \d$ it holds that
\begin{equation}\label{eq5}
 |B_\b\cap \d|<\mu.
\end{equation}
\end{claim}

\begin{proof}[Proof of Claim]
  Suppose to the contrary that $S\su \{\d(\a):\a<\l\}$ is stationary and that 
  $\b(\d)\ge \d$ is chosen  for $\d\in S$ such that  $ |A_{\b(\d)}\cap
  \d|\ge \mu$. We may assume that $\d_1<\d_2\imply \b(\d_1)<\d_2$ for
  $\d_1,\d_2\in S$ by intersecting $S$ with a club, and by thinning $S$ out assume that
  $\k(\beta(\d))=|C_{\b(\d)}|=\k(*)$ is fixed for all $\d\in S$.

The set $S'=S\cap \acc S$ is stationary. If $\d\in S'$ then for
all sufficiently large regular $\k<\mu$ it holds that $M_\d\cap [\d]^{\k}$ is
dense in $[\d]^{\k}$ by Lemma \ref{continuity}, hence, as
$|B_{\b(\d)}\cap \d|\ge \mu$, for each sufficiently large regular $\k<\mu$ there is
$X\in M_\d$ such that $|X|=\k$ and $X\su B_{\b(\d)}$. 

Define $F(\d)=X$, for each $\d\in S'$, such that $X\su B_{\b(\d)}$,
$X\in M_\d$ and $|X|>\k(*)$. By Fodor's Lemma, we can assume that $F$
is fixed on a stationary $S''\su S'$. Fix, then, $\d_1<\d_2$ in $S''$ with
$F(\d_1)=F(\d_2)=X\in M_{\d_1}$. But now $|B_{\b(\d_1)}\cap
B_{\b(\d_2)}|\ge |X|>\k(*)$, and as $\b(\d_1)<\d_2$ 
 this contradicts equation (\ref{eq4}) above.
\end{proof}

The conclusion of the theorem now follows from the induction
hypothesis. 
\end{proof}

In the case that $\rho=\mu=\cf\mu$ the assumption that
every subfamily of $\mathcal G$ of cardinality $\rho$ is $\mu$-e.d is,
of course, superfluous, as $\mathcal G$ is a.d by the
similarity with $\mathcal F$ and every $\rho$-uniform a.d family of
regular size $\rho$ is e.d.

\bigbreak

\noindent\textbf{Concluding remarks.} After Cohen's and Easton's
results showed that the GCH and many of its instances were not
provable in ZFC the impression may have been that the GCH or some
other additional axiom was required for proving general combinatorial
theorems for all infinite cardinals. It now seems that, at least the
combinatorics of splitting families of sets, which for a long time was
hindered by the ``totally independent'' aspect of cardinal arithmetic,
is, in the end, amenable to investigation on the basis of ZFC. We
expect that the method presented here will be useful proving additional
absolute combinatorial relations which generalize known combinatorial
relations between finite cardinals to infinite ones.

\bigbreak
\noindent\textbf{Acknowledgments.} I thank A. Hajnal
for suggesting to me, in the conference celebrating his 80-th
birthday, to apply the methods from \cite{NogaShelah} to splitting
families of sets, and S. Shelah for his proof of Lemma
\ref{continuity} and for his explanations and opinions about the SWH,
which I used in Section \ref{swh}.

Part of the research presented here was done when I was a member of
the Institute for Advanced Study in Princeton during 2010-11. I am
grateful to the IAS for its support and to the School of Mathematics
for its  wonderful atmosphere.

\end{document}